\documentclass{mfat}
\pagespan{311}{329}

\usepackage{mmap}
\usepackage[utf8]{inputenc}

\theoremstyle{plain}
\newtheorem{theorem}{Theorem}
\newtheorem{lemma}[theorem]{Lemma}
\newtheorem{proposition}[theorem]{Proposition}

\newtheorem*{corollary*}{Corollary}

\numberwithin{equation}{section} \numberwithin{theorem}{section}

\begin{document}

\title[Poisson measure as a spectral measure] {Poisson measure as a
  spectral measure of a family of commuting selfadjoint operators,
  connected with some moment problem}

\author{Yu. M. Berezansky}
\address{Institute of Mathematics, National Academy of
Sciences of Ukraine, 3 Teresh\-chenkivs'ka, Kyiv, 01601, Ukraine}
\email{berezansky.math@gmail.com}

 \date{20/07/2016;\ \ Revised
  16/09/2016} \dedicatory{To the memory of my dear daughter
  Natasha. She all the time stands before my eyes.} 


\begin{abstract}
  It is proved that the Poisson measure is a spectral measure of some
  family of commuting selfadjoint operators acting on a space
  constructed from some generalization of the moment problem.
\end{abstract}

\maketitle

\section{Introduction}

In the years 1991--1998 in the works
\cite{Berez-93,Berez-Livin-Lytv-95,Lytv-95,Berez-98} the authors have
constructed a spectral theory of Jacobi fields, which was a
generalization of the spectral theory of one selfadjoint operator
generated by one Jacobi matrix to families of commuting
selfadjoint operators.  In particular, it was shown that the Poisson
measure on a definite infinite-dimensional space can be considered
as a spectral measure for some Jacobi field \cite{Berez-00}.  This
result is deeply connected with article \cite{Ito-Kubo}, see also
recent work \cite{Lytv-15}.

This article also is connected with a study of a Poisson measure as a
spectral measure, but from a different position. In the theory of
point probability processes, a procedure of constructing of Poisson
measure from Lebesgue-Poisson by means of a Kolmogorov-type extension
theorem, see \cite{Part}, is well-known. The Lebesgue-Poisson measure
is given on the space $\Gamma_0(X)$ of finite configurations of points
of the space $X$, but the Poisson measure will be a measure on the
space of infinite configurations $\Gamma(X)$.

For the space of functions on $\Gamma_0(X)$, it is possible to define
the Kondratiev-Kuna convolution $\star $, see \cite{Kondr-Kuna}, which
is a wide generalization of the classical convolution on sequences of
numbers in the theory of classical moment problem. The convolution
$\star $ generates, as in classical power case, in a natural way, a
system of commuting operators on the Hilbert space constructed by
means of a given ``moment" sequence.  It is found that the Poisson
measure will be a spectral measure for this family of operators for the
corresponding moment sequences on the space, which is wider than
$\Gamma(X)$.

It is necessary to explain what measure $\pi$ on the space
$\mathcal{D}'$ is a Poisson measure in our understanding. We assume
that for $\omega\in \mathcal{D}'$ and finite smooth functions $f$ on
$X$ there is a ``pairing", $\left\langle \omega,f\right\rangle$. Then
our measure $\pi$ is a Poisson measure if and only if for its Laplace
transform we have the identity
\begin{equation}\label{Poisson_measure}
\int_{\mathcal{D}'} e^{\left\langle \omega,f\right\rangle}d\pi(\omega)=
\exp{\left(\int_{X}(e^{f(x)}-1)d\sigma(x)\right)},
\end{equation}
where $\sigma$ is some fixed initial (``intensity") measure on $X$
and $f$ is arbitrary.

The article consists of two sections, --- Section
\ref{sect_pois_measure_classic} and Section
\ref{sect_pois_measure_spectr_measure}.  In Section
\ref{sect_pois_measure_classic} we present some results about
configurations and a classical account of introducing a Poisson
measure.  This Section is connected with the works
\cite{Part,Albeverio-98,Kondr-Kuna,Oliv-phd,Fink-doctor} and, of
course, with \cite{Berez-Tesko-16,Berez-Mierz-07,Berez-03}.
Unfortunately, I can not find books or articles that would contain
needed facts about measures and, therefore, it was necessary to write
Section~2.

In Section~\ref{sect_pois_measure_spectr_measure} we introduce a Poisson
measure as a spectral measure $\rho$.  We explain that in this case
$\rho(\Gamma_0(X))>0$.

Therefore, if we understand a Poisson measure as a measure $\pi$ on
$\mathcal{D}'$ with the Laplace transform satisfying
\eqref{Poisson_measure}, then it is possible that $\rho=\pi$ satisfies
the condition $\rho(\Gamma_0(X))>0$ and it is natural to interpret the
set $\Gamma_0(X)$ as a part of $\mathcal{D}'$.  This remark is
connected with the work \cite[p.~12]{Berez-Tesko-16}.  For more details,
see the end of Section \ref{sect_pois_measure_spectr_measure}.

It is necessary to make one refinement to the article
\cite{Berez-Tesko-16}. Namely, the definition of ``non-overlapping
configurations $\gamma$"\ (formula (2.4) in \cite{Berez-Tesko-16}) was
not explained. The results of article \cite{Berez-Tesko-16} are true,
using the notion of the set $\Gamma$ of usual configurations (see
\eqref{Gamma} in Section~\ref{sect_pois_measure_classic} of the
present article).

\section{Poisson measure. Classical account}\label{sect_pois_measure_classic}

At first we recall some definitions and notations from the article
\cite{Berez-Tesko-16}.

Let $X$ be a connected $C^{\infty}$ non-compact Riemannian manifold.
We denote by $\mathcal{D}:=C_{\rm fin}^\infty(X)$ the set of all
real-valued infinitely differentiable functions on $X$ with compact
support.  Denote by $\mathcal{D}_{\mathbb{C}}$ the complexification of
$\mathcal{D}$.  We will consider $\mathcal{D}$ as a nuclear
topological space with the projective limit topology.  Let
$\mathcal{F}_0(\mathcal{D}):=\mathbb{C}$ and
$\mathcal{F}_n(\mathcal{D}):=\mathcal{D}_{\mathbb{C}}^{\mathbin{\widehat{\otimes}}n}$,
$n\in\mathbb{N}$, i. e., it is the space of all complex-valued
symmetric infinitely differentiable functions on $X^n$ with compact
supports and corresponding to the
$\mathcal{D}_{\mathbb{C}}^{\mathbin{\widehat{\otimes}}n}$ topology
(the topology of test functions in theory of generalized functions of
variables from $X$).

Construct the space of finite sequences
\begin{equation}\label{F_fin}
\mathcal{F}_{\operatorname{fin}}(\mathcal{D}):=\bigoplus_{n=0}^\infty \mathcal{F}_n(\mathcal{D})\ni f=(f_0,f_1,\dots), \quad f_n\in \mathcal{F}_n(\mathcal{D}),
\end{equation}
i. e., sequences $f$, for which only a finite number of components $f_n$
are different from zero.  Convergence in this space is equivalent
to uniform finiteness of sequences and coordinate-wise convergence of
every coordinate $f_n$ from the space $\mathcal{F}_n(\mathcal{D})$ in
the topology mentioned above.

Let us recall the notion of the space $\Gamma=\Gamma(X)$ of all
configurations generated by $X$.  It is the set of all locally finite
subsets $\gamma$ of $X$:
\begin{equation}\label{Gamma}
\Gamma:= \left\lbrace \gamma\subset X \left|
|\gamma\cap\Lambda|<\infty \text{ for every compact } \Lambda \subset X \vphantom{\sum A_A}\right. \right\rbrace
\end{equation}
(here $|\cdot|$ is the cardinality of this set).
Each $\gamma \in \Gamma$ consists of distinct points from $X$,
and $\Gamma$ consists of all different configurations $\gamma$
(subsets of $X$).

The topology into the space $\Gamma$ is introduced in the following
way.  Consider the space $\mathcal{D}'$ of continuous linear
functionals on the space $\mathcal{D}$ and the weak topology in
$\mathcal{D}'$ (see, e. g., [6], Chap.~1, \S~1).  Let $\gamma=
[x_1,x_2,\dots,]\in\Gamma$, $x_1,x_2,\dots \in X$ be a certain
configuration.  Denote by $\omega_\gamma$ the corresponding
generalized function from $\mathcal{D}'$:
\begin{equation}\label{w_gamma}
w_\gamma(\varphi):=\left(\sum_{n=1}^\infty\delta_{x_n}\right)(\varphi)=
\sum_{n=1}^\infty\delta_{x_n}(\varphi)
=:\sum_{n=1}^\infty\varphi(x_n) =:\langle\gamma,\varphi\rangle \quad (\varphi\in \mathcal{D})
\end{equation}
(the sum in \eqref{w_gamma} is finite, since $\varphi$ is a finite
function and $x_n$ ``tends to infinity").  Thus, we have a one-to-one
correspondence $\Gamma\ni\gamma\longleftrightarrow \omega_\gamma\in
\mathcal{D}'$, and the weak topology in $\mathcal{D}'$ defines some
topology on $\Gamma$, known as the vague topology (we use this
definition, but usually instead of the space $\mathcal{D}$ we use the
space of finite continuous functions with uniformly finite
convergence).

Denote the set of all finite configurations $\xi=[x_1,\dots,x_n]$,
where $x_1,\dots,x_n\in X$, $x_j\neq x_k$ if $j\neq k$,
$n\in\mathbb{N}$, by $\Gamma_0=\Gamma_0(X)\subset\Gamma(X)=\Gamma$. We
will understand $\Gamma_0$ as a part of the space $\Gamma$ and
topologize it with the relative topology of the space $\Gamma$ (in the
second part of this Section we will introduce another topologization
of the space $\Gamma_0$). We will denote elements of the space
$\Gamma$ by $\gamma,\theta,\dots$ and elements of the space
$\Gamma_0$ by $\xi, \eta,\dots$

It is useful to do the following remarks to the notion of vague
topology.

Suppose we have a given topological space $T$ with neighborhoods $u,
v, \dots$ and some set $A\subset T$.  We introduce a topology into
$A$, using, as a system of neighborhoods, the intersections $u\cap A$,
$v\cap A$, \dots, i.~e., we introduce into $A$ the relative topology.
Let $\overline{A}$ be the closure of $A$ in $T$.  Then the relative
topology in $\overline{A}$ completely defined by the relative topology in
$A$.

In our case we introduce the relative topology into $A=\Gamma_0(X)$,
using the complete space $T=\mathcal{D}'\supset\Gamma_0(X)$.  As follows
from \eqref{w_gamma}, $\Gamma_0(X)$ is a dense set in
$\Gamma(X)=\overline{A}$ in topology of the space $\mathcal{D}'$.  As a
result, the vague topology in $\Gamma_0(X)$ completely defines the vague
topology in $\Gamma(X)$.

In other words, in definition \eqref{w_gamma} it is possible to take
$\gamma$ to be only finite configurations,
$\gamma=\xi=[x_1,\dots,x_m]\in\Gamma_0(X)$, $m\in\mathbb{N}$.  Thus,
finite configurations completely define the topology in $\Gamma(X)$.

But the one-point configurations $[x_1]$, $x_1\in X$, in principle, do
not define this topology uniquely; it is defined by finite sums of
$\delta$-functions.

It is convenient to introduce the following construction, connected
with space $\Gamma_0=\Gamma_0(X)$.  Denote by
$\Gamma^{(n)}=\Gamma^{(n)}_X$ the set of all finite configurations
using the space $X^n$, $n\in\mathbb{N}$;
$\Gamma^{(0)}=\Gamma^{(0)}_X=\varnothing$.

Namely, we understand, by $\Gamma^{(n)}$ for $n\in\mathbb{N}$, a
subset of the space $X^n=(x_1,\dots,x_n)$ of points that are symmetric
(i.~e. the point $(x_1,\dots,x_n)$ does not depend on the order of
$x_1,\dots,x_n\in X$) and their ``coordinate $x_j\in X$'' are
different.

Then the equality of disjunct summands is obvious,
\begin{equation}\label{disjunct_summands}
\Gamma_0=\Gamma_0(X)=\bigsqcup_{n=0}^\infty\Gamma^{(n)}=\bigsqcup_{n=0}^\infty\Gamma^{(n)}_X.
\end{equation}

Let us return to the definition \eqref{F_fin} of the space
$\mathcal{F}_{\operatorname{fin}}(\mathcal{D})$.  Every element of
this space is a finite vector of the following type:
\begin{equation}\label{f_in F_fin}
f=(f_0,f_1,\dots,), \quad f_n\in \mathcal{F}_n(\mathcal{D})=\mathcal{D}_{\mathbb{C}}^{\widehat{\otimes} n},
    \quad  n\in\mathbb{N}; \quad f_0\in\mathbb{C}.
\end{equation}
Fix $n\in\mathbb{N}$.  The component $f_n$ of this vector may be
understood as a smooth complex-valued finite function of point
$\xi\in\Gamma^{(n)}\subset X^n$, which is symmetric.  Conversely, such
function is always a component $f_n$ of a certain vector $f$
\eqref{f_in F_fin}.  For $n=0$ we can understand $f_0\in \mathbb{C}$
as the value in the point $\varnothing$ of such a
function.

Thus, we can understand the vector $f$ from \eqref{f_in F_fin} as a
function
\begin{equation}\label{f}
\Gamma_0=\bigsqcup_{n=0}^\infty\Gamma^{(n)}\ni\xi\mapsto f(\xi)\in\mathbb{C},
\end{equation}
where, for every $n\in\mathbb{N}$, the values of such a function
$\Gamma^{(n)}\ni\xi\mapsto f(\xi)$ are infinitely differentiable
finite symmetric function on $X^n\supset \Gamma^{(n)}$.  For $n=0$ we
have $\varnothing\mapsto f(\varnothing)\in\mathbb{C}$.  Additionally,
every function \eqref{f} must be finite in the ``direction'' $n$,
i.~e., $f(\xi)=0$, where $\xi$ belongs to
$\bigsqcup_{n=m}^\infty\Gamma^{(n)}$ for $m$ large enough.

In what follows, we will identify the vector $f$ of type \eqref{F_fin}
with a corresponding finite function $f(\xi)$ on the space \eqref{f}.
Below, as a rule, the functions of type \eqref{f} will be denoted by
$f, g, \dots:$ $\Gamma_0\ni\xi\mapsto f(\xi)\in\mathbb{C}, \dots $

The above account and formulas \eqref{F_fin}--\eqref{f} are given for
the space $X$.  But we can, instead of $X$, use its subset $Y\subset
X$, topologized with the relative topology.  Then we get the objects
$\Gamma(Y)$, $\Gamma_0(Y)$, $\Gamma^{(n)}(Y)$, $n\in\mathbb{N}_0$.

In this article, as in \cite{Berez-Tesko-16}, the Kondratiev--Kuna
convolution $\star$ \cite{Kondr-Kuna,Oliv-phd} plays an essential
role.

Let us recall the corresponding definitions.  Let us take vectors
\eqref{F_fin}, \eqref{f_in F_fin} in the form of functions on
$\Gamma_0$, $f: \Gamma_0\ni\xi\mapsto f(\xi)\in\mathbb{C}$; $f\in
\mathcal{F}_{\operatorname{fin}}(\mathcal{D})$.  For two such
functions $f$, $g$ we introduce the following convolution:
\begin{align}
(f\star g)(\xi):=\sum_{\xi'\sqcup\xi''=\,\xi}f(\xi')g(\xi'')=&
    \sum_{\xi'\sqcup\xi''\sqcup\xi'''=\,\xi} f(\xi'\sqcup\xi'')g(\xi''\sqcup\xi'''),\label{f_g_convolution}\\
    &\qquad \qquad\qquad \qquad f, g \in \mathcal{F}_{\operatorname{fin}}(\mathcal{D})\notag
\end{align}
(all sums in \eqref{f_g_convolution} are finite).  This convolutions
turns $\mathcal{F}_{\operatorname{fin}}(\mathcal{D})$ into a
commutative algebra with involution
$f=f(\xi)\rightarrow\overline{f(\xi)}=\overline{f}$ and identity
$e=(1,0,0,\dots)$; $e(\varnothing)=1$ and $e(\xi)=0$ for other
$\xi\in\Gamma_0$.  We will denote this algebra by
$\mathcal{A}:\mathcal{F}_{\operatorname{fin}}(\mathcal{D})=\mathcal{A}$.

For the algebra $\mathcal{A}$ it is natural to introduce the notion of
a character $\chi_\varphi(\xi)$.  We will call a character the
following function:
\begin{equation}\label{character_chi}
\chi_\varphi:\Gamma_0\rightarrow\mathbb{R}, \quad
    \xi \rightarrow \chi_\varphi(\xi)=\prod_{x\in\xi}\varphi(x),
    \quad \xi\in\Gamma_0\setminus \varnothing; \quad  \chi_\varphi(\varnothing) = 1,
\end{equation}
where $\varphi\in \mathcal{D}$ is given.  One may easily calculate
that
\begin{equation}\label{character_convolution}
(\chi_\varphi\star \chi_\psi)(\xi)=\chi_{\varphi+\psi+\varphi\psi}(\xi),
    \quad \xi\in\Gamma_0; \quad \varphi,\psi\in \mathcal{D}.
\end{equation}
Let us explain that, in equality \eqref{character_convolution}, we
have the additional member $\varphi+\psi$, since the algebra
$\mathcal{A}$ is an algebra with formally added identity $\varnothing$
in contrast to the ordinary discrete group algebra.

The investigation of convolution $\star $ is deeply connected with the
$K$-transform introduced on the basis of the papers of A.~Lenard by
Yu.~G.~Kondratiev and T.~Kuna in \cite{Kondr-Kuna}.

This transform is a linear operator, acting from
$\mathcal{F}_{\operatorname{fin}}(\mathcal{D})$, which is understood
as a space of functions on $\Gamma_0$, into complex-valued functions
on $\Gamma\supset\Gamma_0$,
\begin{equation}\label{K-trans}
\mathcal{F}_{\operatorname{fin}}(\mathcal{D})\ni f\mapsto(Kf)(\gamma):=\sum_{\xi\subset\gamma}f(\xi)=
    F(\gamma)\in\mathbb{C}\text{ if }\xi\neq\varnothing;
    \quad F(\varnothing) = f(\varnothing).
\end{equation}
The sum in \eqref{K-trans} is finite.  We do not repeat the
definitions, given in \cite{Berez-Tesko-16}, of the spaces between
which the operator $K$ acts.  Such an information about the spaces
under consideration will be given further in this article.

Note that that $K$-transform has an algebraically inverse operator.
Namely, it is easy to prove that if for $f\in
\mathcal{F}_{\operatorname{fin}}(\mathcal{D})$,
\begin{equation}\label{Kf=F}
(Kf)(\gamma)=F(\gamma), \quad \gamma\in \Gamma,
\end{equation}
then
\[(K^{-1}F)(\xi)=\sum_{\eta\in\xi}(-1)^{|\xi\setminus\eta|}F(\eta), \quad \xi\in \Gamma_0. \]

The $K$-transform has also one remarkable property: it maps the
algebra $\mathcal{A}$ into an algebra of functions on $\Gamma$ with
ordinary multiplication, i.~e.,
\begin{equation}\label{A_multiplication}
(K(f\star g))(\gamma)=(Kf)(\gamma)(Kg)(\gamma), \quad \gamma\in\Gamma; \quad f, g \in \mathcal{F}_{\operatorname{fin}}(\mathcal{D})=\mathcal{A}.
\end{equation}

Let us mention another property of the transform $K$, --- it transfers
the functions $f \in \mathcal{F}_{\operatorname{fin}}(\mathcal{D})$
into the functions $(Kf)(\gamma)$, $\gamma\in\Gamma$, that are
cylindrical in some sense.  Such cylindrical property explains the
fact that the inverse operator $K^{-1}$ acts only from a part of
values of a function $F(\gamma)$ (namely, only from their values on
$\Gamma_0\subset\Gamma$), see \eqref{Kf=F}.

Let us explain this assertion.  Every function $f \in
\mathcal{F}_{\operatorname{fin}}(\mathcal{D})$ has the form
\eqref{f_in F_fin} and according to \eqref{F_fin} has a finite number
of coordinates, different from zero.  In the interpretation of this
function $f$ as a function $f(\xi)$ on the space $\Gamma_0$
\eqref{disjunct_summands}, this property means that $f(\xi)$ depends
only on $\xi\in\bigsqcup_{n=0}^m\Gamma^{(n)}$ with some
$m\in\mathbb{N}_0$, i.~e., we have a function
\begin{equation}\label{f_finite_non0_coord}
\bigsqcup_{n=0}^m\Gamma^{(n)}\ni\xi\mapsto f(\xi)\in\mathbb{C}.
\end{equation}
Moreover, every function $\Gamma^{(n)}\ni\xi\mapsto f(\xi)$,
$n=1,\dots,m$, is smooth and finite on $X^n\supset\Gamma^{(n)}$.

Thus, every value in the ``point'' $\gamma$ of the described function
$(Kf)(\gamma)$ by formula \eqref{K-trans} depends only on a sum of the
values $f(\xi)$ with $\xi\subset\gamma$.  Consider the case where the
set $\gamma$ consists of infinitely many different points $x\in X$.
The sum in \eqref{K-trans} is finite and does not depend on values of
our function $\Gamma_0\ni\xi\mapsto f(\xi)\in\mathbb{C}$ in points of
$\gamma\setminus\xi$.  So, the function $(Kf)(\gamma)$ is, in some
sense, ``almost cylindrical''.  Almost, because it is necessary to
consider $\gamma=[x_1,x_2,\dots]\in\Gamma$ as a ``vector with the
coordinates $x\in X$''.  Thus, the function $(Kf)(\gamma)$, with fixed
$f$, depends only on a finite number of ``coordinates $x\in X$''.

Let us prove the following simple fact.

\begin{lemma}\label{lemma_Kf_continuous}
  Let $f \in \mathcal{F}_{\operatorname{fin}}(\mathcal{D})$.  Then the
  function $(Kf)(\gamma)$, $\gamma\in\Gamma$, is continuous in the
  vague topology.
\end{lemma}

\begin{proof}
  Consider fixed $\eta=[x_1,x_2,\dots]\in\Gamma$ and all
  $\xi\in\Gamma_0$, for which $\xi\subset\eta$.  We get the following
  set, which is infinite, if $\eta\in\Gamma\setminus\Gamma_0$:
\begin{equation}\label{eta}
\begin{array}{l}
\left[x_1\right],\left[x_2\right],\dots;\\
\left[x_1,x_2\right],\left[x_1,x_3\right],\dots,\left[x_2,x_3\right],\left[x_2,x_4\right],\dots;\\
\left[x_1,x_2,x_3\right],\left[x_1,x_2,x_4\right],\dots;\\
.\ .\ .\ .\ .\ .\ .\ .\ .\ .\ .\ .\ .\ .\ .\ .\ .\ .\ .\ .\ .\ .\ .\ .\ .\ .\ .\ .\ .\ .
\\
\left[x_1,x_2,\dots,x_n\right],\dots;\\
.\ .\ .\ .\ .\ .\ .\ .\ .\ .\ .\ .\ .\ .\ .\ .\ .\ .\ .\ .\ .\ .\ .\ .\ .\ .\ .\ .\ .\ .  \\
\end{array}
\end{equation}

Consider expression \eqref{K-trans}.  The function $f$ belongs to
$\mathcal{F}_{\operatorname{fin}}(\mathcal{D})$, therefore in its
representation of the form \eqref{f_finite_non0_coord} we have finite
$m\in\mathbb{N}$ and every set $\left\{\xi\in\Gamma^{(n)}\right\}$ is
precompact in the topology of the space $X^n$, $n=1,\dots,m$.

Therefore we have, for fixed
$f\in\mathcal{F}_{\operatorname{fin}}(\mathcal{D})$ and fixed
$\gamma=\left[x_1,x_2,\dots\right]\in\Gamma$,
\begin{equation}\label{Kf(gamma)}
\begin{split}
(Kf)(\gamma)=\sum_{\xi\subset\gamma}f(\xi)=
f(\varnothing)+\sum_{[x_1]\subset\gamma}f([x_1])+
\sum_{[x_1,x_2]\subset\gamma}f([x_1,x_2])\\
+\dots+\sum_{[x_1,\dots,x_n]\subset\gamma}f([x_1,\dots,x_n])+\dots
\end{split}
\end{equation}
In \eqref{Kf(gamma)} every sum in the right-hand side is finite (see
\eqref{eta}).  Moreover, the whole sum is finite, since $f$ has
representation \eqref{f_finite_non0_coord}.

Let $f\in\mathcal{F}_{\operatorname{fin}}(\mathcal{D})$ that is
understood as a finite vector of form \eqref{F_fin},
$f=(f_0,f_1,\dots)$ with components $f_n\in\Gamma^{(n)}$ (see
\eqref{f_finite_non0_coord}).  Then it is possible to calculate, using
\eqref{Kf(gamma)}, the expression $(Kf)([x_1,\dots,x_n])$,
$n\in\mathbb{N}$ (see \cite[p.~209]{Kondr-Kuna}).  Namely, denote by
$\xi_\alpha=[x_{\alpha_1},\dots,x_{\alpha_l}]\in\Gamma_0$, where
$\alpha$ is the set $\{\alpha_1,\dots,\alpha_l\}=:\alpha$ of all
different indexes from $\{1,\dots,n\}$, $1\leq l\leq n$.  Then we have
the representation of $(Kf)(\xi)$ as a finite sum,
\begin{equation}\label{Kf(xi)}
(Kf)([x_1,\dots,x_n])=\sum_{\alpha\subset\{1,\dots,n\}}f(\xi_\alpha),
\quad n\in\mathbb{N}.
\end{equation}

From \eqref{Kf(xi)} it we get a proof of the lemma.  Namely,
\eqref{Kf(xi)} means that $(Kf)(\xi)$ is continuous, when
$\xi=[x_1,\dots,x_n]$ varies over $X^n$ without a diagonal
$x_1=\dots=x_n$, since every function $f(\xi_\alpha)$ from
\eqref{Kf(xi)} has such a property in the corresponding space $X^l$,
$1\leq l\leq n$.

Configuration $\gamma\in\Gamma$ in the function $(Kf)(\gamma)$,
\eqref{Kf(gamma)}, has the form $\gamma=$
$$[x_1,\dots,x_m,x_{m+1}^0,x_{m+2}^0,\dots]$$ 
with fixed
$m\in\mathbb{N}$ and some fixed points $x_{m+1}^0,x_{m+2}^0,\dots$,
tending ``to infinity''. Therefore, the mentioned continuity of
$(Kf)(\gamma)$ gives continuity of this function in the vague
topology.
\end{proof}

We need once more to turn to the space
$\mathcal{F}_{\operatorname{fin}}(\mathcal{D})$. This is a linear
complex space with the topology described after definition
\eqref{F_fin}.  It is useful to know that this space (with its
topology) is a projective limit of some Hilbert spaces,
\begin{equation}\label{F_fin_pr_lim}
\mathcal{F}_{\operatorname{fin}}(\mathcal{D})=\operatorname{pr}\lim_{\tau\in T,\,p\geq 1} F(H_\tau,p)=
    \bigcap_{\tau\in T,\,p\geq 1}F(H_\tau,p).
\end{equation}

Here $F(H_\tau,p)$ is the weighted Fock space which consists of
sequences $f=(f_n)_{n=0}^\infty$, $f_n\in
H_{\tau,\mathbb{C}}^{\widehat{\otimes}}=:F_n(H_\tau)$ such that
\begin{equation}\label{F_norm}
\|f\|_{F(H_\tau,p)}^2=\sum_{n=0}^\infty\|f\|_{F_n(H_\tau)}^2 p_n < \infty
\end{equation}
with the corresponding scalar product.  Here $p=(p_n)_{n=0}^\infty$,
$p_n\geq 1$ means a number weight, $H_{\tau,\mathbb{C}}$ is the
complexification of the Sobolev space $W_2^{\tau_1}(X, \tau_2(x)dm(x))$,
where $\tau=(\tau_1,\tau_2(x))$, $\tau_1\in\mathbb{N}_0$,
$\tau_2(x)\geq 1$ is a $C^\infty$ weight and $m$ is Riemannian measure
on $X$.

The space $\mathcal{F}_{\operatorname{fin}}(\mathcal{D})$
\eqref{F_fin} is a nuclear space; the embedding of the spaces
$F(H_\tau,p)$ with corresponding $\tau$ and $p$ is of Hilbert-Schmidt
type.  For these notions, we refer to, e.~g., \cite{Berez-Kondr-95}.

As we have noted, the space
$\mathcal{F}_{\operatorname{fin}}(\mathcal{D})$ is a commutative
algebra $\mathcal{A}$ with respect to multiplication $\star $, is
endowed with involution ``$-$'', and has a unit element $e$.

Construct the Hilbert space connected with the algebra
$\mathcal{A}=\mathcal{F}_{\operatorname{fin}}(\mathcal{D})$.  Namely,
consider a linear functional $s\in\mathcal{A}'$, which is called
non-negative, if
\begin{equation}\label{nonnegative_functional}
s(f\star \overline{f})\geq 0, \quad f\in \mathcal{A}.
\end{equation}
Any non-negative functional $s\neq 0$ generates the following
quasi-scalar product on $\mathcal{A}$
\begin{equation}\label{quasiscalar_product}
(f,g)_{\mathcal{H}_s}=s(f\star\overline{g}), \quad f,g
\in\mathcal{A}.
\end{equation}

Identifying every $f\in \mathcal{A}$ such that $s(f\star\overline{f})=0$ with zero,
considering corresponding classes and completing the space of these classes,
we construct a Hilbert space $\mathcal{H}_s$.

In this article we will consider only the case where
\begin{equation}\label{positive_functional}
\left\{f\in\mathcal{A}\left|s(f\star\overline{f})=0\vphantom{\bigcup}\right.\right\}=0,
\end{equation}
i.~e., the positive (non-degenerate) case.  In this case
\eqref{quasiscalar_product} is a scalar product on $\mathcal{A}$ and
the completion of $\mathcal{A}$ with respect to this scalar product
gives our space $\mathcal{H}_s$.

We finish the first part of this Section, devoted to definitions and
facts, which are necessary for following account.

The second part, devoted to the classical account of a Poisson
measure, we start with a remark that, for this purpose, it is
necessary to change the topology of the space $\Gamma$, which we have
introduced above by \eqref{Gamma}.

Namely, we have $\Gamma=\Gamma_0\bigcup(\Gamma\setminus\Gamma_0)$ and
the topology of the part $\Gamma_0$ must be another one; the space
$\Gamma_0$ is represented by \eqref{disjunct_summands} as a disjoint
sum of the spaces $\Gamma^{(n)}$, $n\in\mathbb{N}_0$, each of which is
a subspace of $X^n\supset\Gamma^{(n)}$ endowed with the relative
topology; $\Gamma^{(0)}=\varnothing$. The convergence in the space
$\Gamma_0$ \eqref{disjunct_summands} is that of uniform finiteness and
coordinate-wise convergence for every coordinate $f_n$ of a vector
$f=(f_0,f_1,\dots)\in\bigsqcup_{n=0}^\infty\Gamma^{(n)}$.  We
will call this topology on $\Gamma_0$ an ``ordinary topology''.

The topology of $\Gamma\setminus\Gamma_0$ is the previous vague
topology, understanding it as the relative topology on
$\Gamma\setminus\Gamma_0\subset\Gamma$ with the vague topology
considered on $\Gamma$. Thus, we have to the end of this Section the
representation of $\Gamma$ as a union of two disjoint topological
spaces
\begin{equation}\label{Gamma_2}
\Gamma=\Gamma_0\sqcup(\Gamma\setminus\Gamma_0),
\end{equation}
where $\Gamma_0$ is topologized with the stated above ordinary
topology and $\Gamma\setminus\Gamma_0$ with the vague topology.  We
will say in this case that $\Gamma$ \eqref{Gamma_2} is topologized
with the ordinary-vague topology.

Let us introduce the classical Lebesgue-Poisson measure.  We start
with a fixed measure $\sigma$ on the $\sigma$-algebra $\mathcal{B}(X)$
of Borel sets from the space $X$ with a topology given on this space.
So we have the measure
\begin{equation}\label{sigma}
\mathcal{B}(X)\ni\alpha\mapsto\sigma(\alpha)\geq 0.
\end{equation}

This measure \eqref{sigma} must be non-degenerate, i.~e.,
$\sigma(\alpha)>0$ for every open set $\alpha\subset X$ and
non-atomic, i.~e., for every $x\in X$ $\sigma(\{x\})=0$.  Assume that
$\sigma(X) = +\infty$.  We will call this measure $\sigma$ initial or
an intensity measure.

We fix some $n\in\mathbb{N}$ and denote by $\sigma^{(n)}$ or
$\sigma^{\widehat{\otimes}n}=\sigma\times\dots\times\sigma$ ($n$
times) the symmetric tensor product of the measure $\sigma$.  This
measure $\sigma^{(n)}$ is defined on the $\sigma$-algebra of Borel
sets $\mathcal{B}(X^n)$ as follows:
\begin{equation}\label{sigma^n}
\mathcal{B}(X^n)\ni\alpha^{(n)}\mapsto\sigma^{(n)}(\alpha^{(n)}):=
    \sigma^{\widehat{\otimes}n}(\alpha^{(n)})\geq 0.
\end{equation}

Thus, this measure $\sigma^{(n)}$ is also defined on the sets from
space $\Gamma^{(n)}=\Gamma^{(n)}_X$ (see \eqref{disjunct_summands}),
which belong to $\mathcal{B}(X^n)$.

For $n=0$ we put $\sigma^{(0)}(\varnothing)\geq 0$.

The Lebesgue-Poisson measure $d\lambda(\xi)$ on every set
$\alpha\subset\Gamma_0=\Gamma_0(X)$, which is Borel with respect to the
ordinary topology, i.~e. $\alpha\in \mathcal{B}(\Gamma_0)$, is defined
by  the formula
\begin{equation}\label{B(Gamma_0)}
\begin{split}
\mathcal{B}(\Gamma_0)=\mathcal{B}(\Gamma_0(X))\ni\alpha\mapsto\lambda(\alpha)
&
=
  \sigma^{(0)}(\varnothing)+
    \frac{1}{1!}\sigma^{(1)}(\alpha^{(1)})+
    \frac{1}{2!}\sigma^{(2)}(\alpha^{(2)})
    \\
   &  +\dots+
    \frac{1}{n!}\sigma^{(n)}(\alpha^{(n)})+\dots,
\end{split}
\end{equation}
where $\alpha^{(n)}=\alpha\cap\Gamma^{(n)}$.

Thus, if we have a function $\Gamma_0\ni\xi\mapsto f(\xi)\in\mathbb{C}$,
which is integrable with respect to $d\lambda(\xi)$, then
\begin{equation}\label{int_f_on_Gamma_0}
\int_{\Gamma_0}f(\xi)d\lambda(\xi)=f(\varnothing)\sigma^{(0)}(\varnothing)+
\sum_{n=1}^\infty\frac{1}{n!}\int_{\Gamma^{(n)}}f(x_1,\dots,x_n)d\sigma^{(n)}(x_1,\dots,x_n).
\end{equation}

In what follows, we set $\sigma^{(0)}(\varnothing)=1$ in
\eqref{B(Gamma_0)}, \eqref{int_f_on_Gamma_0}. From \eqref{B(Gamma_0)},
\eqref{int_f_on_Gamma_0} it is easy to conclude that
$\lambda(\Gamma_0(Y))=e^{\sigma(Y)}$, where $Y\subset X$ is a compact
subset of $X$.

After having introduced the Lebesgue-Poisson measure, it is possible
to introduce a Poisson measure on the space $\Gamma=\Gamma(X)$.
This measure is given on the $\sigma$-algebra of Borel sets $\mathcal{B}(\Gamma(X))$
in the ordinary-vague topology.

Recall that this topology on $\Gamma(X)$ is the weak topology in
$\mathcal{D}'$, where $\mathcal{D}=C_{\rm fin}^\infty(X)$ with the
standard topology in theory of generalized functions on $X$.  The
space $\Gamma(X)$ is included into $\mathcal{D}'$ by
$\Gamma(X)\ni\gamma=[x_1,x_2,\dots]\mapsto\sum_{n=1}^\infty\delta_{x_n}$,
where $\delta_{x}$ denotes the $\delta$-function at the point $x\in X$
(see \eqref{w_gamma}).

Denote by $\mathcal{B}_c(X)$ the sets from $\mathcal{B}(X)$ with
compact closures, where $\mathcal{B}(X)$ is the $\sigma$-algebra of
Borel sets in the topology of $X$.  Let $Y\in \mathcal{B}_c(X)$ be
such \textit{fixed} set.  Consider the set of configurations
$\Gamma(Y)$.  By definition \eqref{Gamma} this set $\Gamma(Y)$
consists only of finite configurations $\gamma$, since closure $Y$
in $X$ is compact.

We introduce the vague topology in $\Gamma(Y)$ replacing the space $X$ with
$Y$, i.~e., this topology is the weak topology in the space
$\mathcal{D}'_Y$, where $\mathcal{D}_Y=C_{\rm fin}^\infty(Y)$.

It is possible to prove, that this topology is as type of topology in
the space \eqref{disjunct_summands}, i.~e.,
\begin{equation}\label{Gamma(Y)}
\Gamma(Y)=\bigsqcup_{n=0}^\infty\Gamma_Y^{(n)}=\Gamma_0(Y),
\end{equation}
where $\Gamma_Y^{(n)}$ is the set of all configurations of type
$[x_1,\dots,x_n]$, where $x_j\in Y$, $1\leq j\leq n$.  In
\eqref{Gamma(Y)} we have a disjoint sum of $\Gamma_Y^{(n)}$ with the
ordinary topology in each $\Gamma_Y^{(n)}$.  Such a result follows
from the Lemma 2.5 and its proof can be found in
\cite{Berez-Tesko-16}.

Thus, it is reasonable to write in \eqref{Gamma(Y)}, instead of
$\Gamma(Y)$, the symbol $\Gamma_0(Y)$; $Y\in \mathcal{B}_c(X)$.

Note that every function $f\in C_{\rm fin}^\infty(Y)$ ca be extended
by zero to a function on $C_{\rm fin}^\infty(X)$.  So, we have some
system of topologies on $X$, depending on $Y\in \mathcal{B}_c(X)$ (but
in this system the set $X\setminus Y$ is considered as one point of the
space).

Introduce the ``projection'' $p_Y$ of a ``point'' from $\Gamma(X)$ to
a ``point'' from $\Gamma_0(Y)$,
\begin{equation}\label{p_Y}
\Gamma(X)\supset\Gamma_0(Y), \quad \Gamma(X)\ni \gamma\mapsto
p_Y\gamma=\gamma_Y:= \gamma\bigcap Y\in \Gamma_0(Y);
\end{equation}
evidently, if $\Gamma_0(Y_1)\supset\Gamma_0(Y_2),$ $Y_2 \subset Y_1$, $Y_1, Y_2 \in \mathcal{B}_c(X)$,
then
$\Gamma_0(Y_1)\ni \gamma\mapsto p_{Y_2,Y_1}=\gamma_{Y_2}:= \gamma\cap Y_2\in \Gamma_0(Y_2)$;
$p_Y=p_{YX}.$

Denote by $p_Y^{-1}\gamma$, $\gamma\in\Gamma_0(X)$, and by
$p_{Y_2,Y_1}^{-1}\gamma$, $\gamma\in\Gamma_0(Y_1)$ the full
preimage of $p_Y$ and $p_{Y_2;Y_1}$, defined by \eqref{p_Y}. It is
evident that
\begin{equation}\label{p_Y^-1}
\begin{split}
&\Gamma_0(Y)\ni \gamma\mapsto p_Y^{-1}\gamma=\theta\in\Gamma(X)\text{ such that }\theta_Y=\gamma;\\
&\Gamma_0(Y_2)\ni \gamma\mapsto p_{Y_2,Y_1}^{-1}\gamma=\theta\in\Gamma_0(Y_1)
\text{ such that }\theta_{Y_2}=\gamma.\\
\end{split}
\end{equation}

We will apply the operators $A$ of type $p_Y$, $p_{Y_2,Y_1}$, $p_Y^{-1}$, $p_{Y_2,Y_1}^{-1}$
from \eqref{p_Y} and \eqref{p_Y^-1} to different sets of configurations $\gamma\in\alpha$.
Of course, such an application $A\alpha$ means that $A\alpha:=\{A_\gamma, \gamma\in\alpha\}$,
i.~e., the application is pointwise.

Consider an arbitrary $\Gamma_0(Y)$, where $Y\in\mathcal{B}_c(X)$ with the
corresponding ordinary topology in $\Gamma_0(Y)$ and the $\sigma$-algebra
of Borel sets $\mathcal{B}(\Gamma_0(Y))$ in this topology.  We have
$\mathcal{B}(\Gamma_0(Y))\subset \mathcal{B}(\Gamma(X))$,
$\mathcal{B}(\Gamma_0(Y_2))\subset \mathcal{B}(\Gamma_0(Y_1))$ and the
relations \eqref{p_Y^-1} can be rewritten as follows:
\begin{equation}\label{beta}
\begin{split}
&\mathcal{B}(\Gamma_0(Y))\ni\alpha\mapsto \beta\in\mathcal{B}(\Gamma(X))
\text{ such that }p_Y^{-1}\beta=:\beta_Y=\alpha;\\
&\mathcal{B}(\Gamma_0(Y_2))\ni\alpha\mapsto \beta\in\mathcal{B}(\Gamma_0(Y_1))
\text{ such that }p_{Y_2,Y_1}^{-1}\beta=:\beta_{Y_2}=\alpha.\\
\end{split}
\end{equation}

We can construct, due to \eqref{p_Y}, \eqref{p_Y^-1}, \eqref{beta}, a
projective limit $\operatorname{prlim}_Y\Gamma_0(Y)$ of such spaces
with the corresponding $\sigma$-algebras $\mathcal{B}(\Gamma_0(Y))$
(see, e.~g. \cite{Part,Fink-doctor}).

Now we will define a projective limit on $\Gamma(X)$ of probability measures on $\Gamma_0(Y)$,
where $Y\in \mathcal{B}_c(X)$ are arbitrary.
Such a family of measures $\mu^Y$ is defined on the $\sigma$-algebra $\mathcal{B}(\Gamma_0(Y))$; the
measure $\mu^Y$ must be a probability measure, i.~e., $\mu^Y(\Gamma_0(Y))=1$, $Y\in \mathcal{B}_c(X)$.
Their projective limit, $\mu^X$, is a measure on $\mathcal{B}(\Gamma(X))$
and must also be a probability, $\mu^X(\Gamma(X))=1$.

This family of measures $\mu^Y$ is supposed to be consistent in the
following sense, see \eqref{beta}:
\begin{equation}\label{mu_consistent}
\mu^{Y_2}(\alpha)=\mu^{Y_1}(p_{Y_2,Y_1}^{-1}\alpha), \quad
\alpha\in\mathcal{B}(\Gamma_0(Y_2))\subset\mathcal{B}(\Gamma_0(Y_1))
\end{equation}
for every $Y_2\subset Y_1$; $Y_1, Y_2\in \mathcal{B}_c(X)$;
i.~e., $\mu^{Y_2}(\beta_{Y_2})=\mu^{Y_1}(\beta)$, $\beta\in\mathcal{B}(\Gamma_0(Y_1))$.

For the limit measure $\mu^X$ in the projective limit, we must also have
the following property:
\begin{equation}\label{measure_prlim}
\mu^Y(\alpha)=\mu^X(p_Y^{-1}\alpha), \quad
\alpha\in\mathcal{B}(\Gamma(Y))\subset\mathcal{B}(\Gamma(X)),
\end{equation}
$Y\subset X$; $Y\in \mathcal{B}_c(X)$ (see \eqref{beta});
i.~e., $\mu^Y(\beta_Y)=\mu^{X}(\beta)$, $\beta\in\mathcal{B}(\Gamma(X))$.

The question is whether there exists such a projective limit of
measures.  The answer is given by the corresponding version of a
Kolmogorov-type theorem (see
\cite{Part,Albeverio-98,Fink-doctor,Oliv-phd}).

\begin{theorem}\label{th_proj_limit}
  Suppose that there exists a consistent family of probability
  measures $\mu^Y$, $Y\in \mathcal{B}_c(X)$, on the $\sigma$-algebras
  $\mathcal{B}(\Gamma_0(Y))$, i.~e., \eqref{mu_consistent} is
  satisfied.

  Then there exists a unique probability measure $\mu^X$ on
  $\mathcal{B}(\Gamma(X))$ such that for every $Y\in \mathcal{B}_c(X)$
  we have \eqref{measure_prlim}.
\end{theorem}

After these general results about the projective limit of the spaces
$\Gamma(Y)$ endowed with corresponding Borel $\sigma$-algebras
$\mathcal{B}(\Gamma_0(Y))$, $Y\in \mathcal{B}_c(X)$, we can introduce
a Poisson measure with the a use of the Lebesgue-Poisson measure.

\begin{theorem}\label{th_Poisson_measure}
  Consider $Y\in \mathcal{B}_c(X)$, and the probability measures 
\begin{equation}\label{prob_measures}
\pi^Y(\alpha)=e^{-\sigma(Y)}\lambda(\alpha), \quad \alpha\in\mathcal{B}(\Gamma_0(Y)),
\end{equation}
on the $\sigma$-algebra $\mathcal{B}(\Gamma_0(Y))$, where
$d\lambda(\xi)$ is Lebesgue-Poisson measure \eqref{B(Gamma_0)} on
$\Gamma_0(Y)$ with the initial measure $\sigma$ \eqref{sigma} given on
$\mathcal{B}(X)$.

Using the above mentioned Theorem~\ref{th_proj_limit} we can conclude
that, on $\mathcal{B}(\Gamma(X))$, there exists a unique probability
measure $\pi$, for which
\begin{equation}\label{limit_measure_pi}
\begin{split}
\pi^Y(\alpha)=\pi(p_Y^{-1}\alpha), \quad \alpha\in\mathcal{B}(\Gamma(Y)), \quad Y\in \mathcal{B}_c(X); \\
\quad \text{ i. e. } \pi^Y(\beta_Y)=\pi(\beta), \quad \beta \in \mathcal{B}(\Gamma(X)),
\end{split}
\end{equation}
where the projection $p_Y$ is given by \eqref{p_Y}, \eqref{p_Y^-1},
\eqref{beta}.  Such a measure is called a Poisson measure.
\end{theorem}
\begin{proof}
  Using \eqref{B(Gamma_0)} we conclude that, in our situation,
  $\lambda(\Gamma_0(Y))=e^{\sigma(Y)}$, therefore the measure
  \eqref{prob_measures} is a probability measure, $\pi^Y(\Gamma_0(Y))=1$.

  The condition \eqref{mu_consistent} is also fulfilled.  It is
  necessary to prove that the measure $\mu^Y$ of the form
  $\mu^Y=\pi^Y$, where $\pi^Y$ is given by \eqref{prob_measures},
  satisfies equality \eqref{mu_consistent}.  We have
\[ \mu^{Y_2}(\alpha)=e^{-\sigma(Y_2)}\lambda(\alpha), \quad
\alpha\in\mathcal{B}(\Gamma_0(Y_2))\subset \mathcal{B}(\Gamma_0(Y_1)), \]
\[ \mu^{Y_1}(p_{Y_2,Y_1}^{-1}\alpha)=e^{-\sigma(Y_1)}\lambda(p_{Y_2,Y_1}^{-1}\alpha) =
e^{-\sigma(Y_1)}\lambda(\beta).\]
Here $\beta\in\mathcal{B}(\Gamma_0(Y_1))$ is such that $\alpha=\beta_{Y_2}=p_{Y_2,Y_1}\beta$.

Thus, it is necessary to prove that
\begin{equation}\label{th_Poisson_measure_proof_1}
e^{-\sigma(Y_1)}\lambda(\beta)=e^{-\sigma(Y_2)}\lambda(\beta_{Y_2}),
\quad \beta\in\mathcal{B}(\Gamma_0(Y_1)), \quad Y_2\subset Y_1\subset X.
\end{equation}

The measure $\sigma$ on $X$ is given by \eqref{sigma} and is
non-degenerate and non-atomic.  The proof of
\eqref{th_Poisson_measure_proof_1} in the case of the Lebesgue measure
$\sigma$ on $X=\mathbb{R}^d$ is given in \cite{Fink-doctor}.  The
general situation is considered in \cite{Part,Albeverio-98}.
\end{proof}

Consider some properties of the Poisson measure, which will be needed
in the sequel.

\begin{lemma}\label{lemma_Poisson_measure_positive}
  The Poisson measure is positive on open sets from $\Gamma(X)$ in the
  vague topology.
\end{lemma}

\begin{proof}
  Let $\beta$ be an open set in vague topology from $\Gamma(X)$.  It
  is necessary to prove that $\pi(\beta)>0$.  If
  $\beta\subset\Gamma(X)$ is an open set, then it is easy to prove
  that, for every $Y\subset X$, the set
  $\alpha:=\beta_Y=p_Y^{-1}\beta$ is also open in the vague topology,
  considered on the space $\Gamma(Y)$.

  In particular, this is true if $Y\in \mathcal{B}_c(X)$.  But using
  \eqref{Gamma(Y)} we know that the vague topology on
  $\Gamma(Y)=\Gamma_0(Y)$ is the ordinary topology.  Thus, the set
  $\alpha$ is open in the ordinary topology on $\Gamma_0(Y)$.  Using
  the formula \eqref{limit_measure_pi} we assert that it is necessary
  to prove that $\pi^Y(\alpha)>0$, or, using \eqref{prob_measures},
  that $\lambda(\alpha)>0$.  But this follows from \eqref{B(Gamma_0)}.
\end{proof}

Let us mention some simple properties of the Poisson and
Lebesgue-Poisson measures introduced by
Theorem~\ref{th_Poisson_measure} and by definition \eqref{B(Gamma_0)}.

The Lebesgue-Poisson measure $\lambda$ is defined on sets
$\alpha\subset\Gamma(X)$ which are Borel in the vague topology,
$\alpha\in\mathcal{B}(\Gamma(X))$.  Let
\begin{equation}\label{LP_meaure_1}
X=\bigcup_{n=1}^\infty Y_n, \quad Y_1\subset Y_2\subset \dots, \quad Y_n\in \mathcal{B}_c(X).
\end{equation}
Then
\begin{equation}\label{LP_meaure_2}
\Gamma_0(X)=\bigcup_{n=1}^\infty \Gamma_0(Y_n),
\quad \Gamma_0(Y_1)\subset \Gamma_0(Y_2)\subset \dots, \quad \Gamma_0(Y_n)\in \mathcal{B}(\Gamma(X)).
\end{equation}

Using \eqref{LP_meaure_2} for every
$\alpha\in\mathcal{B}(\Gamma(X))$ we get
\begin{equation}\label{LP_meaure_3}
\alpha=\bigcup_{n=1}^\infty\alpha_n, \quad
\text{where}\quad  \alpha_n=\alpha\bigcap\Gamma_0(Y_n)\in
\mathcal{B}(\Gamma(X)); \quad \alpha_1\subset \alpha_2 \subset
\dots
\end{equation}

Absolute additivity of the Poisson measure $\pi$, with a use of
\eqref{LP_meaure_3} and \eqref{prob_measures}, gives
\begin{equation}\label{LP_meaure_4}
\pi(\alpha)=\lim_{n\rightarrow\infty}\pi(\alpha_n)=
\lim_{n\rightarrow\infty} e^{-\sigma(Y_n)}\lambda(\alpha_n).
\end{equation}
\begin{lemma}\label{lemma_Poisson_measure_zero}
  The Poisson measure $\pi$ of the set of all finite configurations is
  equal to zero, $\pi(\Gamma_0(X))=0$.
\end{lemma}

\begin{proof}
Let $\Lambda\subset X$ be some compact subset of $X$ and
$\Gamma(\Lambda)$ a corresponding subset of $\Gamma(X)$,
$\Gamma(X)\supset\Gamma(\Lambda)$. All configurations from
$\Gamma(\Lambda)$ are finite, therefore we can write (see
\eqref{disjunct_summands})
\begin{equation}\label{Gamma(Lambda)}
\Gamma(\Lambda)=\Gamma_0(\Lambda)=\bigsqcup_{m=1}^\infty\Gamma^{(m)}_\Lambda.
\end{equation}

From \eqref{LP_meaure_1} it follows that $\Lambda\subset Y_{n_0}$ for
some $n_0\in\mathbb{N}$. Therefore, $\Gamma_0(\Lambda)\subset
\Gamma_0(Y_{n_0})\subset \Gamma_0(Y_{n_0+1})\subset\dots$ Take $\alpha
= \Gamma_0(\Lambda)\in \mathcal{B}(\Gamma(X))$ in
\eqref{LP_meaure_3}. Then, in this case,
$\alpha_{n_0}=\alpha_{n_0+1}=\alpha_{n_0+2}=\dots$ and
\eqref{LP_meaure_4} gives
\[ \pi(\Gamma_0(\Lambda))=\pi(\alpha)=
\lim_{n\rightarrow\infty} e^{-\sigma(Y_n)}\lambda(\alpha_n)=
\lambda(\alpha_{n_0})\lim_{n\rightarrow\infty} e^{-\sigma(Y_n)}=0, \]
since  \eqref{LP_meaure_1} takes place, and $\sigma(Y_n)\rightarrow+\infty$.

Therefore, $\pi(\Gamma_0(\Lambda))=0$ for every $\Lambda\subset X$,
i.~e. $\pi(\Gamma_0(X))=0$.
\end{proof}

We will also prove two known facts about the measures under
consideration.

\begin{theorem}\label{th_Laplace_transform}
  The Laplace transform of the Poisson measure $\pi(\alpha)$,
  $\alpha\in\mathcal{B}(\Gamma(X))$ is given by
\begin{equation}\label{Laplace_transform}
\int_{\Gamma(X)}e^{\langle\gamma,f\rangle}d\pi(\gamma)=
\exp\left(\int_{X}(e^{f(x)}-1)d\sigma(x)\right), \quad f\in\mathcal{D}.
\end{equation}
\end{theorem}
\begin{proof}
  For fixed $f\in\mathcal{D}$ we can find a set $Y\in
  \mathcal{B}_c(X)$ such that $f(x)=0$, $x\in X\setminus Y$. Using
  \eqref{limit_measure_pi}, \eqref{prob_measures}, \eqref{sigma^n} and
  \eqref{int_f_on_Gamma_0} we can write
\begin{align*}
\int_{\Gamma(X)}e^{\langle\gamma,f\rangle}d\pi(\gamma) & =
\int_{\Gamma(Y)}e^{\langle\gamma,f\rangle}d\pi^Y(\gamma)=
e^{-\sigma(Y)}\int_{\Gamma(Y)}e^{\langle\gamma,f\rangle}d\rho(\gamma)
\\
& =
e^{-\sigma(Y)}\sum_{n=0}^\infty\frac{1}{n!}\int_{Y^n}e^{\sum_{j=1}^nf(x_j)}
d\sigma^{(n)}(x_1,\dots,x_n)
\\
& =
e^{-\sigma(Y)}\sum_{n=0}^\infty\frac{1}{n!}\left(\int_{Y}e^{f(x)}d\sigma(x)\right)=
\exp\left(\int_{X}(e^{f(x)}-1)d\sigma(x)\right).
\end{align*}
\end{proof}

It is usual to say that, if for some measure $\rho$, it is possible to
define its Laplace transform, and for this transform equality
\eqref{Laplace_transform} it follows, that this measure $\rho$ is a
Poisson measure.

At last, we will need the following equality that is a special case of
Theorem~4.1 from \cite[Example 4.1]{Kondr-Kuna} (see also
\cite{Oliv-phd}).
\begin{proposition}\label{th_LP_P}
  The following relation between the Lebesgue-Poisson measure
  $d\lambda(\xi)$ and Poisson measure $d\pi(\gamma)$ holds true:
\begin{equation}\label{LP_P}
\int_{\Gamma_0(X)}f(\xi)d\lambda(\xi)=\int_{\Gamma(X)}(Kf)(\gamma)d\pi(\gamma),
\quad f\in\mathcal{F}_{\operatorname{fin}}(\mathcal{D}).
\end{equation}
\end{proposition}

\section{Poisson measure as a spectral measure of some family of
  commutating selfadjoint operators}
\label{sect_pois_measure_spectr_measure}

In the first part of Section \ref{sect_pois_measure_classic} we had a
Hilbert space $\mathcal{H}_s$ which was constructed in the following
way.

We have introduced a commutative algebra $\mathcal{A}$, whose elements
are vectors from the space
$\mathcal{F}_{\operatorname{fin}}(\mathcal{D})$ (see \eqref{F_fin},
\eqref{f_in F_fin}, \eqref{f}) and a composition $\star$ defined by
\eqref{f_g_convolution}.

On the space
$\mathcal{A}=\mathcal{F}_{\operatorname{fin}}(\mathcal{D})$, we
consider a linear functional $s\in\mathcal{A}'$, $s\neq 0$, which is
non-negative in the sense of \eqref{nonnegative_functional}. We
consider now only the case when $s$ is positive, i.~e., condition
\eqref{positive_functional} is fulfilled. Construct the Hilbert space
$\mathcal{H}_s$ that is a completion of $\mathcal{A}$ with respect to
the scalar product \eqref{quasiscalar_product}.

For this space $\mathcal{H}_s$ in the article \cite{Berez-Tesko-16},
we considered a family $(A(\varphi))_{\varphi\in \mathcal{D}}$ of
unbounded (in general) operators defined by
\begin{equation}\label{A_phi}
\mathcal{H}_s\supset \mathcal{F}_{\operatorname{fin}}(\mathcal{D})=\mathcal{A}\ni f\mapsto A(\varphi)f=\varphi\star f\in\mathcal{A},
\end{equation}
where $\varphi$ is a function from $\mathcal{D}$ (i.~e., a real-valued
function from $\mathcal{F}_{1}(\mathcal{D})=\mathcal{D}$).  The
closure $\tilde{A}(\varphi)$ of operator \eqref{A_phi} is well-defined
in the space $\mathcal{H}_s$ and is Hermitian.

Such operators $A(\varphi)$, $\varphi\in \mathcal{D}$, were
investigated in the article \cite{Berez-Tesko-16} (and earlier in
\cite{Berez-03,Berez-Mierz-07}) even in the general case, when the
requirement of positivity \eqref{positive_functional} was omitted, and
the Hilbert space $\mathcal{H}_s$ consisted of classes of vectors from
$\mathcal{A}$.

Under some conditions every operator $A(\varphi)$, $\varphi\in
\mathcal{D}$, is essentially selfadjoint and their set forms a set of
commutative selfadjoint operators acting on the space $\mathcal{H}_s$.
In the article \cite{Berez-Tesko-16} (and in
\cite{Berez-03,Berez-Mierz-07}) the spectral representation for this
family $(\tilde{A}(\varphi))_{\varphi\in \mathcal{D}}$ was considered
and some applications of this theory were given.

In this article, we will consider only the case where the functional
$s\in\mathcal{A}'$ has the form of an integral. Namely, let $s$ be
given by the integral
\begin{equation}\label{s_as_integral}
s(f)=\int_{\Gamma_0}f(\xi)d\nu(\xi)=
\sum_{n=0}^\infty\,\,\int_{\Gamma^{(n)}}f(\xi)d\nu(\xi),
\quad f\in
\mathcal{A}=\mathcal{F}_{\operatorname{fin}}(\mathcal{D}),
\end{equation}
where $d\nu(\xi)$ is some finite measure on the $\sigma$-algebra of
Borel sets in ordinary topology $\Gamma_0$, given by \eqref{Gamma_2}.
For $f\in \mathcal{F}_{\operatorname{fin}}(\mathcal{D})$ the vectors
\eqref{f_in F_fin} are finite and every function
$f\upharpoonright\Gamma^{(n)}$ is a finite smooth function, therefore,
the integral \eqref{s_as_integral} always exists.

In the articles \cite{Berez-03,Berez-Mierz-07,Berez-Tesko-16} the
following essential fact was proved: if the measure $d\nu(\xi)$ in the
representation \eqref{s_as_integral} is such, that for every compact
$\Lambda\subset X$ there exists a constant $C_\Lambda>0$ such that
\begin{equation}\label{nu_leq_c}
\nu(\Gamma_\Lambda^{(n)})\leq C_\Lambda^n, \quad n\in\mathbb{N}_0,
\end{equation}
then the closures $\tilde{A}(\varphi)$ of the operators $A(\varphi)$
on the space $\mathcal{H}_s$ make a family
$(\tilde{A}(\varphi))_{\varphi\in \mathcal{D}}$ of commuting
selfadjoint operators (this result is true even when the condition
\eqref{positive_functional} is not fulfilled).  Let us explain that
$\Gamma_\Lambda^{(n)}$ denotes the space $\Gamma_X^{(n)}$ from
\eqref{disjunct_summands} if we replace $X$ with $\Lambda\subset X$.

Let us pass to a study of the case of a Poisson measure.  Recall that
a non-atomic initial measure
$\mathcal{B}(X)\ni\alpha\mapsto\sigma(\alpha)\geq 0$ is given on Borel
sets of the space $X$.  Using this measure by rule \eqref{B(Gamma_0)}
we construct the corresponding Lebesgue-Poisson measure $\lambda(\xi)$
on the $\sigma$-algebra of Borel sets $\mathcal{B}(\Gamma_0)$
with respect to the ordinary topology on~$\Gamma_0$.

\begin{theorem}\label{th_positive_functional}
  The functional $s$ of the form \eqref{s_as_integral}, where
  $d\nu(\xi)=d\lambda(\xi)$ is a Lebesgue-Poisson measure on
  $\Gamma_0$, is positive, i.~e., condition
  \eqref{positive_functional} is fulfilled.

  The condition \eqref{nu_leq_c} for such a functional is also
  fulfilled.
\end{theorem}
\begin{proof}
  Let $f\in\mathcal{A}=\mathcal{F}_{\operatorname{fin}}(\mathcal{D})$
  and $s(f\star \overline{f})=0$.  Then it is necessary to prove that
  $f=0$.

  Denote $g=f\star \overline{f}\in
  \mathcal{F}_{\operatorname{fin}}(\mathcal{D})$.  Using
  \eqref{A_multiplication} we conclude that
\begin{equation}\label{th_positive_functional_proof_1}
(Kg)(\gamma)=(K(f\star \overline{f}))(\gamma)=(Kf)(\gamma)(K\overline{f})(\gamma)=\left|(Kf)(\gamma)\right|^2,
\quad \gamma\in\Gamma.
\end{equation}

We apply Proposition~\ref{th_LP_P} and the corresponding equality
\eqref{LP_P} to a vector $g\in
\mathcal{F}_{\operatorname{fin}}(\mathcal{D})$.  We get, using
\eqref{s_as_integral} with the measure $d\lambda(\xi)$ and
\eqref{th_positive_functional_proof_1} that
\begin{equation}\label{th_positive_functional_proof_2}
\begin{split}
s(f\star \overline{f})=s(g) &
=\int_{\Gamma_0}g(\xi)d\lambda(\xi)=
\int_{\Gamma}(Kg)(\gamma)d\pi(\gamma)
\\
& =\int_{\Gamma}\left|(Kf)(\gamma)\right|^2d\pi(\gamma)=
\int_{\Gamma\setminus\Gamma_0}\left|(Kf)(\gamma)\right|^2d\pi(\gamma).
\end{split}
\end{equation}
Here $\lambda(\xi)$ is the Lebesgue-Poisson measure on the Borel
$\sigma$-algebra of the space $\Gamma_0$ with the ordinary topology
and $d\pi(\gamma)$ is a Poisson measure on the Borel $\sigma$-algebra
of the space $\Gamma$ topologized by the ordinary-vague topology (see
\eqref{Gamma_2}).  Since $\pi(\Gamma_0)=0$, we have
\eqref{th_positive_functional_proof_2}.

Let $f\in \mathcal{F}_{\operatorname{fin}}(\mathcal{D})$ be such that
$s(f\star \overline{f})=0$.  Then we conclude from
\eqref{th_positive_functional_proof_2} that
\begin{equation}\label{Kf(gamma)=0}
(Kf)(\gamma)=0
\end{equation}
for almost all $\gamma\in\Gamma\setminus\Gamma_0$ with respect to a
Poisson measure on $\Gamma\setminus\Gamma_0$.  According to
Lemma~\ref{lemma_Kf_continuous}, the function $(Kf)(\gamma)$,
$\gamma\in\Gamma$, is continuous with respect to the vague topology.
On other hand, the Poisson measure is positive on open sets from
$\Gamma$ in the vague topology
(Lemma~\ref{lemma_Poisson_measure_positive}).  Therefore the equality
\eqref{Kf(gamma)=0} means that $(Kf)(\gamma)$ is equal to zero for
every $\gamma\in\Gamma\setminus\Gamma_0$.  But the transform $K$ has
an algebraically inverse operator $K^{-1}$ (see \ref{Kf=F}), hence
$f=0$.

Pass to the second part of the Theorem.  The Lebesgue-Poisson measure
on the space $\Gamma_0$ is defined by the series (see \eqref{B(Gamma_0)})
\begin{equation}\label{LP_measure_by_series}
\mathcal{B}(\Gamma_0)\ni\alpha\mapsto
\sum_{n=0}^\infty\frac{1}{n!}\sigma^{(n)}(\alpha^{(n)})=\lambda(\alpha),
\end{equation}
where $\mathcal{B}(\Gamma_0)$ is the $\sigma$-algebra of Borel sets in
the ordinary topology $\Gamma_0$ and $\sigma^{(n)}(\alpha^{(n)})$ are
values of the symmetric tensor product $m^{\widehat{\otimes}n}$ of the
measure $m$ on $X$ on the set $\alpha^{(n)}:=\alpha\cap\Gamma^{(n)}$,
$n\in\mathbb{N}_0$ (see \eqref{disjunct_summands}).

For the bounded function $\Gamma_0\ni\xi\mapsto f(\xi)\in\mathbb{C}$,
measurable with respect to $\mathcal{B}(\Gamma_0)$, we have (see
\eqref{int_f_on_Gamma_0}) that
\begin{equation}\label{th_positive_functional_proof_3}
\begin{split}
\int_{\Gamma_0}f(\xi)d\lambda(\xi)=
    &\sum_{n=0}^\infty\frac{1}{n!}
    \int_{\Gamma^{(n)}}f_n(x_1,\dots,x_n)d\sigma^{(n)}(x_1,\dots,x_n),\\
&\qquad f_n(x_1,\dots,x_n)=f(\xi)\upharpoonright\Gamma^{(n)}.
\end{split}
\end{equation}

Let $\Lambda\subset X$ be an arbitrary compact set. Then, similarly to
\eqref{disjunct_summands}, \eqref{th_positive_functional_proof_3}, we
have for a bounded measurable function $f(\xi)$,
$\xi\in\Gamma_0(\Lambda)$, that
\begin{equation}\label{th_positive_functional_proof_4}
\begin{split}
&\Gamma_0(\Lambda)=\bigsqcup_{n=0}^\infty\Gamma^{(n)}_\Lambda,\\
& \int_{\Gamma_0(\Lambda)}|f(\xi)|d\lambda(\xi)=
\sum_{n=0}^\infty\frac{1}{n!}
    \int_{\Gamma^{(n)}_\Lambda}|f_n(x_1,\dots,x_n)|d\sigma^{(n)}(x_1,\dots,x_n).
\end{split}
\end{equation}

In particular, for $f(\xi)=1$, $\xi\in\Gamma_0(\Lambda)$, we have
\begin{equation}\label{th_positive_functional_proof_5}
\sum_{n=0}^\infty\frac{1}{n!}\,\lambda(\Gamma^{(n)}_\Lambda)=\lambda(\Gamma_0(\Lambda))<\infty.
\end{equation}
From \eqref{th_positive_functional_proof_5} and \eqref{B(Gamma_0)}
we easily conclude that there exists a certain constant
$C_\Lambda>0$ such that for every $n\in\mathbb{N}_0$,
$\lambda(\Gamma^{(n)}_\Lambda)<C^n_\Lambda$. Thus, the estimate
\eqref{nu_leq_c} is proved.
\end{proof}

For us it is necessary to repeat some main results of the spectral
theory for a family $(\tilde{A}(\varphi))_{\varphi\in \mathcal{D}}$ of
commuting selfadjoint operators on the space $\mathcal{H}_s$, see
\cite[Theorem~5.3, Condition~3.5]{Berez-Tesko-16} and estimate
\eqref{th_positive_functional_proof_5}.

\begin{proposition}\label{prop_A(phi)selfadjoint}
  Let the conditions \eqref{positive_functional} and \eqref{nu_leq_c}
  for the functional $s$ of the form \eqref{s_as_integral} be
  fulfilled.  Then the operators $\tilde{A}(\varphi)$ of the family
  $(\tilde{A}(\varphi))_{\varphi\in \mathcal{D}}$ are selfadjoint in
  the space $\mathcal{H}_s$ and commuting.  This family generates a
  Fourier transform $I$ of the following form:
\begin{equation}\label{I}
\begin{split}
\mathcal{F}_{\operatorname{fin}}(\mathcal{D})\ni
f=(f_n)_{n=0}^\infty\mapsto (If)(\omega) & =:
\widehat{f}(\omega)= (f,P(\omega))_{\mathcal{F}(H)}
\\
& =
\sum_{n=0}^\infty(f_n,P_n(\omega))_{\mathcal{F}_n(H)}\in
L^2(\mathcal{D}',d\rho(\omega)).
\end{split}
\end{equation}

Here $\rho$ is the spectral measure of the family, being a probability
Borel measure on the space $\mathcal{D}'$ of generalized functions
$\omega$ with weak topology, i.~e., on the $\sigma$-algebra
$\mathcal{B}(\mathcal{D}')$.  The closure $\widetilde{I}$ by
continuity of the operator $I$ is a unitary operator between the spaces
$\mathcal{H}_s$ and $L^2(\mathcal{D}',d\rho(\omega))$.  It maps each
operator $\tilde{A}(\varphi)$ into an operator of multiplication by
the function $\left\langle\omega,\varphi\right\rangle$.
\end{proposition}

In \eqref{I}, $P(\omega)=(P_n(\omega))_{n=0}^\infty$, where the functions
$\mathcal{D}'\ni\omega\mapsto
P_n(\omega)\in\left(\mathcal{D}^{\widehat{\otimes}n}\right)'$,
$n\in\mathbb{N}_0$, are similar to polynomials of the first kind in the
classical moment problem; $P_0(\omega)=1$, $\omega\in \mathcal{D}'$.
They satisfy the following equality:
\begin{equation}\label{P(w)}
(P(\omega),A(\varphi)f)_{\mathcal{F}(H)}=\left\langle\omega,\varphi\right\rangle(P(\omega),f)_{\mathcal{F}(H)},
\quad \varphi\in \mathcal{D},  \quad \omega\in \mathcal{D}', \quad
f\in \mathcal{F}_{\operatorname{fin}}(\mathcal{D}).
\end{equation}
Here $\mathcal{F}(H)$ is the usual symmetric Fock space, constructed
from the space $$H=L^2(X,dm(x)),$$ i.~e.,
\[ \mathcal{F}(H)=\bigoplus_{n=0}^\infty \mathcal{F}_n(H). \]

The equality \eqref{P(w)} means that $P(\omega)$ is a joint
generalized eigenvector for the family
$(\tilde{A}(\varphi))_{\varphi\in \mathcal{D}}$ of the operators
$\tilde{A}(\varphi)$ with the eigenvalue
$\left\langle\omega,\varphi\right\rangle$.

Note that, to prove Proposition \ref{prop_A(phi)selfadjoint}, i.~e.,
Theorem 5.3 from \cite{Berez-Tesko-16}, it is necessary to construct
some quasi-nuclear rigging of the space $\mathcal{H}_s$.  This rigging
is constructed by means of spaces \eqref{F_norm}, for details see in
\cite{Berez-Tesko-16,Berez-Mierz-07,Berez-03} and the book
\cite{Berez-Kondr-95}.

Let us pass to some results from \cite{Berez-Tesko-16,Berez-Mierz-07},
that are more deeply connected with a Poisson measure as the spectral
measure.  At first, we will give some results that are stated in the
article \cite[Theorem 4.1]{Berez-Mierz-07} (see also \cite[Section
6]{Berez-Tesko-16}).
\begin{proposition}\label{prop_e}
For any $\omega\in \mathcal{D}'$, consider the function
\begin{equation}\label{e}
e^{\left\langle\omega,\log(1+\varphi)\right\rangle},
\end{equation}
where $\varphi\in \mathcal{D}$ and $\varphi(x)>-1$, $x\in X$.
This function can be decomposed into a series in tensor powers $\varphi^{\otimes n}$
in the following way:
\begin{equation}\label{e_series}
e^{\left\langle\omega,\log(1+\varphi)\right\rangle}=
\sum_{n=0}^\infty\left(\varphi^{\otimes n},P_n(\omega)\right)_{\mathcal{F}_n(H)},
\end{equation}
where the coefficients of this decomposition are just $P_n(\omega)$
from  \eqref{P(w)}.
\end{proposition}

Let $\psi\in \mathcal{D}$ be arbitrary. Then the function
$e^{\psi(x)}-1$ belongs to $\mathcal{D}$ and its values are greater
than $-1$.  Therefore, we can take this function to be $\varphi(x)$ in
the expression \eqref{e}.  Thus, it is possible to write, for $\psi\in
\mathcal{D}$,
\begin{equation}\label{e^psi}
X\ni x\mapsto\varphi(x)=e^{\psi(x)}-1\in \mathcal{D}; \quad
\varphi(x)>-1, \quad x\in X; \quad \psi\in \mathcal{D} \text{ is
arbitrary}.
\end{equation}

Using this change \eqref{e^psi} of the function $\varphi$ to $\psi$,
we can rewrite the equality \eqref{e_series} in the form
\begin{equation}\label{e_series_psi}
e^{\left\langle \omega,\psi\right\rangle}=
\sum_{n=0}^\infty\left(\left(e^{\psi}-1\right)^{\otimes
n},P_n(\omega)\right)_{\mathcal{F}_n(H)}, \quad\psi\in
\mathcal{D},\quad  \omega\in \mathcal{D}'.
\end{equation}

Let $\varphi\in \mathcal{D}$ be arbitrary.  Recall that we have
introduced the notion of a character $\chi_\varphi$, $\varphi\in
\mathcal{D}$, by means of the identity \eqref{character_chi}.  This
definition is of type \eqref{f}, i.~e., we are given some function on
$\Gamma_0$.  But such a function can be given as a sequence of type
\eqref{f_in F_fin}, instead of \eqref{f}.  So, we have the following
definition of the character $\chi_\varphi$:
\begin{equation}\label{character_chi_prod}
\begin{split}
\chi_\varphi&(\xi)=\prod_{x\in\xi}\varphi(x), \quad
\xi\in\Gamma_0\setminus\varnothing;
\quad\chi_\varphi(\varnothing)=1. \\ & \vrule height 1.5ex depth
0pt \, \vrule height 1.5ex depth
0pt
\\
(1,\varphi,&\varphi^{\otimes 2},\dots,\varphi^{\otimes n},\dots); \quad \varphi\in \mathcal{D}.
\end{split}
\end{equation}

Therefore, the right-hand side of the equality \eqref{e_series} can be
understood as the right-hand side of the equality \eqref{I} with
$f_n=\varphi^{\otimes n}$.  Thus, if we prove that, in the case of the
Hilbert space $\mathcal{H}_s$ (constructed from the functional $s$ of
the form \eqref{s_as_integral} with Lebesgue-Poisson measure
$d\nu(\xi)=d\lambda(\xi)$) the vector \eqref{character_chi_prod}
$\chi_\varphi$ belongs to the space $\mathcal{H}_s$, then it is
possible to understand \eqref{I} as the Fourier transform
$\widetilde{I}$ of a vector $\chi_\varphi\in \mathcal{H}_s$.  We will
prove this actually simple fact.

\begin{lemma}\label{lemma_chi_in_H_s}
  Let a functional $s$ have the form \eqref{LP_measure_by_series} with
  the Lebesgue-Poisson measure $d\nu(\xi)=d\lambda(\xi)$.  Then an
  arbitrary character $\chi_\varphi$, $\varphi\in \mathcal{D}$,
  belongs to the space $\mathcal{H}_s$.
\end{lemma}

\begin{proof}
  Using the equality \eqref{character_convolution} for the character
  \eqref{character_chi_prod} we get
\begin{equation}\label{lemma_chi_in_H_s_proof_1}
(\chi_\varphi\star \chi_\varphi)(\xi)=\chi_{2\varphi+\varphi^2}(\xi), \quad \xi\in\Gamma_0; \quad
2\varphi(x)+\varphi^2(x)=:\theta(x), \quad x\in X,\quad  \theta\in \mathcal{D}.
\end{equation}
Denote by $\Lambda\subset X$ the compact set, for which $\theta(x)=0$,
$x\in X\setminus\Lambda$. Then similarly to
\eqref{th_positive_functional_proof_4} we have with some $c\in
(0,\infty)$ that
\begin{equation}\label{lemma_chi_in_H_s_proof_2}
\begin{split}
q(\chi_\varphi) &
:=\int_{\Gamma_0}|\chi_\theta(\xi)|d\lambda(\xi)=
\int_{\Gamma_0(\Lambda)}|\chi_\theta(\xi)|d\lambda(\xi)
\\
& =\sum_{n=0}^\infty\frac{1}{n!}
    \int_{\Gamma^{(n)}_\Lambda}|\theta(x_1),\dots,\theta(x_n))|d\sigma^{(n)}(x_1,\dots,x_n)
\\
&
=\sum_{n=0}^\infty\frac{1}{n!}\left(\int_{\Lambda}|\theta(x)|d\sigma(x)\right)^n
<c<\infty.
\end{split}
\end{equation}

Introduce the notion of a subcharacter, $\chi_{\varphi,\operatorname{sub};k}(\xi)$, $k\in\mathbb{N}_0$:
instead of \eqref{character_chi_prod} we put:
\begin{equation}\label{subcharacter_chi_prod}
\chi_{\varphi,\operatorname{sub};k}(\xi)=
(1,\varphi,\varphi^{\otimes 2},\dots,\varphi^{\otimes k},0,0,\dots)\in \mathcal{F}_{\operatorname{fin}}(\mathcal{D}),
\quad k\in\mathbb{N}_0.
\end{equation}
For subcharacters, the formula \eqref{lemma_chi_in_H_s_proof_2}
$q(\chi_{\varphi,\operatorname{sub};k})$ has the form
\eqref{lemma_chi_in_H_s_proof_2}, but the summation is carried out up
to $k$.  Of course,
\begin{equation}\label{lim_q}
\lim_{k\rightarrow\infty}q(\chi_{\varphi,\operatorname{sub};k})=q(\chi_\varphi),\quad
q(\chi_{\varphi,\operatorname{sub};k})\leq q(\chi_{\varphi,sub;k+1})\leq c.
\end{equation}

We have, for $\varphi\in \mathcal{D}$ (see \eqref{lemma_chi_in_H_s_proof_1}), that
\begin{equation}\label{chi_norm}
\begin{split}
\|\chi_\varphi\|_{\mathcal{H}_s}^2 & = s(\chi_\varphi\star
\overline{\chi_\varphi})= \int_{\Gamma_0}(\chi_\varphi\star
\chi_\varphi)d\lambda(\xi)
\\
&  =\int_{\Gamma_0}\chi_\theta(\xi)d\lambda(\xi)\leq
\int_{\Gamma_0}|\chi_\theta(\xi)|d\lambda(\xi)=
q(\chi_\varphi)
\end{split}
\end{equation}
and a similar identity for
$\chi_{\varphi,\operatorname{sub};k}(\xi)$. From \eqref{lim_q},
\eqref{chi_norm} we conclude that, in the space $\mathcal{H}_s$,
$\lim_{k\rightarrow\infty}\chi_{\varphi,\operatorname{sub};k}=\chi_\varphi$.
But $\chi_{\varphi,\operatorname{sub};k}\in
\mathcal{F}_{\operatorname{fin}}(\mathcal{D})$, therefore,
$\chi_\varphi\in \mathcal{H}_s$.
\end{proof}

Let us repeat that we will apply Proposition~\ref{prop_e}, the main
result of the spectral theory developed in
\cite{Berez-Tesko-16,Berez-Mierz-07,Berez-03,Berez-Kondr-95}.
The corresponding functional $s$ has the form
\begin{equation}\label{s_e}
s(f)=\int_{\Gamma_0}f(\xi)d\lambda(\xi), \quad
f\in\mathcal{A}=\mathcal{F}_{\operatorname{fin}}(\mathcal{D}),
\end{equation}
where $d\lambda(\xi)$ is the Lebesgue-Poisson measure on $\Gamma_0$.

Consider the series \eqref{I} (see also \eqref{e^psi}),
\begin{equation}\label{e_F_trans}
\begin{split}
&\sum_{n=0}^\infty((e^\psi-1)^{\otimes n},
P_n(\omega))_{\mathcal{F}_n(H)}=
\sum_{n=0}^\infty(\varphi^{\otimes n},
P_n(\omega))_{\mathcal{F}_n(H)}, \quad \psi\in \mathcal{D}, \quad
\omega\in \mathcal{D}';
\\ &\qquad \varphi(x)=e^{\psi(x)}-1\in
\mathcal{D}, \quad \varphi(x)>-1, \quad x\in X.
\end{split}
\end{equation}

As follows from Lemma~\ref{lemma_chi_in_H_s}, $\chi_\varphi\in H_s$
and, therefore, the series \eqref{e_F_trans} can be regarded as the
Fourier transform \eqref{I} $(\widetilde{I}\chi_\varphi)(\omega)$ of
the function $e^{\langle \omega,\psi\rangle}$ with a fixed $\omega\in
\mathcal{D}'$, see \eqref{e_series_psi}.

Consider the spectral measure $d\rho(\omega)$, $\omega\in
\mathcal{D}'$ of the family $(\tilde{A}(\varphi))_{\varphi\in
  \mathcal{D}}$ of the operators $\tilde{A}(\varphi)$, which is
defined on the $\sigma$-algebra $\mathcal{B}(\mathcal{D}')$ of Borel
sets in the weak topology. Using \eqref{e_F_trans} we can write
\begin{equation}\label{e_F_trans_2}
\begin{split}
&0\leq
e^{\langle\omega,\psi\rangle}=\sum_{n=0}^\infty(\varphi^{\otimes
n},P_n(\omega))_{\mathcal{F}_n(H)}=
(\widetilde{I}\chi_\varphi)(\omega),\quad \psi\in
\mathcal{D},\quad \omega\in \mathcal{D}'; \\ &\qquad
\qquad\varphi(x)=e^{\psi(x)}-1, \quad x\in X.
\end{split}
\end{equation}

Integrate the equality \eqref{e_F_trans_2} with respect to $\omega\in
\mathcal{D}'$ in measure $d\rho(\omega)$.  We get
\begin{equation}\label{integration_e_by_d_rho}
\int_{\mathcal{D}'}e^{\langle\omega,\psi\rangle}d\rho(\omega)=
\int_{\mathcal{D}'}(\widetilde{I}\chi_\varphi)(\omega)d\rho(\omega)\qquad \varphi\in \mathcal{D}.
\end{equation}
Formally, such an integral can be equal to $+\infty$: we integrate a
non-negative measurable function \eqref{e_F_trans_2} with respect to a
positive finite measure.  But we now prove that the integral
\eqref{integration_e_by_d_rho} is equal to
$s(\chi_\varphi)\in(0,+\infty)$.

It is easy to prove the following general fact.
\begin{lemma}\label{lemma_s_in_(0,+infty)}
  Let the conditions \eqref{nonnegative_functional} and positivity
  \eqref{positive_functional} be fulfilled for the functional $s(f)$,
  $f\in\mathcal{A}=\mathcal{F}_{\operatorname{fin}}(\mathcal{D})$.
  Therefore it is possible to introduce the space $\mathcal{H}_s$.
  Then the functional $s$ is continuous with respect to the norm of
  space $\mathcal{H}_s$ and it is possible to extend it to the whole
  space $\mathcal{H}_s$.
\end{lemma}
\begin{proof}
  Using the Cauchy-Bunyakovski inequality, for $f,g\in\mathcal{A}$, we
  can write
\begin{equation}\label{CB_inequality}
|s(f\star \overline{g})|^2=|(f,g)_{\mathcal{H}_s}|^2\leq \|f\|_{\mathcal{H}_s}\|g\|_{\mathcal{H}_s}.
\end{equation}
Let $g=e$, where $e$ is the unit element of the algebra $\mathcal{A}$.
Then from \eqref{CB_inequality} we have
\[|s(f)|^2=|s(f\star \overline{e})|^2\leq \|f\|_{\mathcal{H}_s}\|e\|_{\mathcal{H}_s} \leq C\|f\|_{\mathcal{H}_s},
\quad C=\|e\|_{\mathcal{H}_s}.\]
\end{proof}
\begin{lemma}\label{lemma_scalar_prod_in_H_s}
The following equality is true:
\begin{equation}\label{scalar_prod_in_H_s}   
s(f)=\int_{\mathcal{D}'}(\widetilde{I}f)(\omega)d\rho(\omega),\quad
f\in \mathcal{H}_s.
\end{equation}
\end{lemma}
\begin{proof}
  Using the Proposition~\ref{prop_e} we can assert that the operator
  $\widetilde{I}$ is a unitary operator between the spaces
  $\mathcal{H}_s$ and $L^2(\mathcal{D}',d\rho(\omega))$.  Thus
\begin{equation}\label{lemma_scalar_prod_in_H_s_proof_1}
(f,g)_{\mathcal{H}_s}=\int_{\mathcal{D}'}(\widetilde{I}f)(\omega)\overline{(\widetilde{I}g)(\omega)}d\rho(\omega),
\quad f,g\in \mathcal{H}_s.
\end{equation}
Let $g=e$ in \eqref{lemma_scalar_prod_in_H_s_proof_1},
i.~e., $e=(1,0,0,\dots)$ or, in the form of a function on $\Gamma_0\ni\xi$,
$e(\xi)=1$ if $\xi=\varnothing$ and $0$ for other $\xi$.  We then have,
instead of \eqref{lemma_scalar_prod_in_H_s_proof_1}, that
\begin{equation}\label{lemma_scalar_prod_in_H_s_proof_2}
(f,e)_{\mathcal{H}_s}=\int_{\mathcal{D}'}(\widetilde{I}f)(\omega)d\rho(\omega),\quad
f\in \mathcal{H}_s.
\end{equation}

If $f\in \mathcal{F}_{\operatorname{fin}}(\mathcal{D})$, then
$(f,e)_{\mathcal{H}_s}=s(f\star \overline{e})=s(f)$.  But using
Lemma~\ref{lemma_s_in_(0,+infty)} we can assert that the functional
$s$ is continuous also on the space $\mathcal{H}_s\supset
\mathcal{F}_{\operatorname{fin}}(\mathcal{D})$.  Thus we can write
$(f,e)_{\mathcal{H}_s}=s(f)$ also for $f\in \mathcal{H}_s$.  Then the
equality \eqref{lemma_scalar_prod_in_H_s_proof_2} gives
\eqref{scalar_prod_in_H_s}.
\end{proof}

Let us return to the equality \eqref{integration_e_by_d_rho}.  Since
$\forall\varphi\in \mathcal{D}$, the character $\chi_\varphi\in
\mathcal{H}_s$ (Lem\-ma~\ref{lemma_chi_in_H_s}), according to
Lemma~\ref{lemma_scalar_prod_in_H_s} we have the following
\emph{essential equality}:
\begin{equation}\label{integration_e_by_d_rho=s}
\int_{\mathcal{D}'}e^{\langle\omega,\psi\rangle}d\rho(\omega)=
\int_{\mathcal{D}'}(\widetilde{I}\chi_\varphi)(\omega)d\rho(\omega)=s(\chi_\varphi),
\quad \varphi\in \mathcal{D}.
\end{equation}

\begin{theorem}\label{th_A(phi)_spectral_representation}
  Consider the family $(\tilde{A}(\varphi))_{\varphi\in \mathcal{D}}$
  of commuting selfadjoint operators $\tilde{A}(\varphi)$ of the
  form \eqref{A_phi} on the space $\mathcal{H}_s$.  This space is
  constructed using the functional $s$ of the form
  \eqref{quasiscalar_product} and \eqref{s_e}, where $d\lambda(\xi)$
  is a Lebesgue-Poisson measure.  These operators $\tilde{A}(\varphi)$
  are indeed selfadjoint and commuting.

  The corresponding spectral representation has the form \eqref{I},
  where $P(\omega)$ is defined by the equation \eqref{P(w)}.  The
  spectral measure $d\rho(\omega)$ is a non-negative measure on the
  space $\mathcal{D}'$ and is given on sets of the Borel
  $\sigma$-algebra constructed with respect to the weak topology on
  $\mathcal{D}'$.

  This measure $d\rho(\omega)$ is Poisson in the following sense: its
  Laplace transform has the form
\begin{equation}\label{d_rho_F-trans}
\int_{\mathcal{D}'}e^{\langle\omega,f\rangle}d\rho(\omega)=
\exp\left(\int_{X}(e^{f(x)}-1)d\sigma(x)\right), \quad f\in\mathcal{D},
\end{equation}
where $d\sigma(x)$ is the initial measure on $X$.
\end{theorem}
\begin{proof}
Consider equality \eqref{e_F_trans_2}.
In this equality $\omega$ is a linear functional $f$ from $\mathcal{D}'$.
In particular, we can take $\omega$ to be
$\gamma\in\Gamma$, where $\gamma=[x_1,x_2,\dots]$, $x_m\in X$ and
\begin{equation}\label{gamma,psi}
\langle\gamma,\psi\rangle:=\omega_\gamma(\psi)=\sum_{m=1}^\infty\psi(x_m), \quad\psi\in \mathcal{D}.
\end{equation}
I.~e., we identify, as usual, $\gamma$ with $\sum_{m=1}^\infty\delta_{x_m}$
(see \eqref{w_gamma}).

As a result we have, using \eqref{gamma,psi}, for the character
$\chi_\varphi$,
\begin{equation}\label{th_A(phi)_spectral_representation_proof_1}
0\leq e^{\langle\gamma,\psi\rangle}=\sum_{n=0}^\infty(\varphi^{\otimes n},P_n(\gamma))_{\mathcal{F}_n(H)}=
(\widetilde{I}\chi_\varphi)(\gamma),
\end{equation}
where $ \psi\in \mathcal{D}$ and $\varphi(x)=e^{\psi(x)}-1$, $x\in X.$

Consider in \eqref{th_A(phi)_spectral_representation_proof_1}, instead
of $\chi_\varphi$, the corresponding subcharacter
\eqref{subcharacter_chi_prod}
$\chi_{\varphi,\operatorname{sub};k}(\xi)\in
\mathcal{F}_{\operatorname{fin}}(\mathcal{D})$.  We can write
\eqref{th_A(phi)_spectral_representation_proof_1} in the form
\begin{equation}\label{th_A(phi)_spectral_representation_proof_2}
\begin{split}
0\leq
e^{\langle\gamma,\psi\rangle}=\sum_{n=0}^\infty(\varphi^{\otimes
n},P_n(\gamma))_{\mathcal{F}_n(H)} & =
\lim_{k\rightarrow\infty}\sum_{n=0}^k(\varphi^{\otimes
n},P_n(\gamma))_{\mathcal{F}_n(H)}
\\
&
=\lim_{k\rightarrow\infty}(\widetilde{I}\chi_{\varphi,\operatorname{sub};k})(\gamma)=
\lim_{k\rightarrow\infty}(K\chi_{\varphi,\operatorname{sub};k})(\gamma).
\end{split}
\end{equation}
We have used in \eqref{th_A(phi)_spectral_representation_proof_2} the inclusion
$\chi_{\varphi,\operatorname{sub};k}\in \mathcal{F}_{\operatorname{fin}}(\mathcal{D})$ and the following equality
for $f\in \mathcal{F}_{\operatorname{fin}}(\mathcal{D})$:
\begin{equation}\label{th_A(phi)_spectral_representation_proof_3}
(\widetilde{I}f)(\gamma)=(If)(\gamma)=(Kf)(\gamma)
\end{equation}
(the equality \eqref{th_A(phi)_spectral_representation_proof_3}
follows from \cite[Lemma~6.3]{Berez-Tesko-16}).

Consider the connection \eqref{LP_P} between the Lebesgue-Poisson measure
$d\lambda(\xi)$ and the Poisson measure $d\pi(\gamma)$,
\begin{equation}\label{d_lambda_and d_pi}
s(f)=\int_{\Gamma_0}f(\xi)d\lambda(\xi)=\int_\Gamma (Kf)(\gamma)d\pi(\gamma),
\end{equation}
where $f\in \mathcal{F}_{\operatorname{fin}}(\mathcal{D})$.

Integrating \eqref{th_A(phi)_spectral_representation_proof_3} with
respect to $\gamma\in\Gamma$ in the measure $d\pi(\gamma)$ and using
\eqref{d_lambda_and d_pi} we get
\begin{equation}\label{th_A(phi)_spectral_representation_proof_4}
\int_\Gamma e^{\langle\gamma,\psi\rangle}d\pi(\gamma)=
\lim_{k\rightarrow\infty}\int_\Gamma(K\chi_{\varphi,\operatorname{sub};k})(\gamma)d\pi(\gamma)=
\lim_{k\rightarrow\infty}s(\chi_{\varphi,\operatorname{sub};k})=s(\chi_\varphi).
\end{equation}
It is easy to prove that one can pass to the limit under the integral.
The last limit exists since $\chi_\varphi\in \mathcal{H}_s$ (see
Lemma~\ref{lemma_s_in_(0,+infty)}).

For the classical Poisson measure $d\pi(\gamma)$ we have the general
equality (see \eqref{Laplace_transform}):
\begin{equation}\label{th_A(phi)_spectral_representation_proof_5}
\int_\Gamma e^{\langle\gamma,\psi\rangle}d\pi(\gamma)=
\exp\left(\int_X(e^{\psi(x)}-1)d\sigma(x)\right), \quad \psi \in \mathcal{D}.
\end{equation}
From equalities \eqref{integration_e_by_d_rho=s},
\eqref{th_A(phi)_spectral_representation_proof_4} and
\eqref{th_A(phi)_spectral_representation_proof_5} we conclude that
\[ 
\pushQED{\qed}
\int_{\mathcal{D}'}e^{\langle\omega,f\rangle}d\rho(\omega)=
\exp\left(\int_{X}(e^{f(x)}-1)d\sigma(x)\right), \quad \psi \in \mathcal{D}.
\qedhere
\popQED \]
	\renewcommand{\qed}{}	
\end{proof}

Let us come back to the question the arose in article
\cite{Berez-Tesko-16}, at the end of Section~2.  It is known, that the
Poisson measure $\pi$ constructed in a classical way from a
Lebesgue-Poisson measure by means of a Kolmogorov-type theorem has the
following property:
\begin{equation}\label{pi=0}
\pi(\Gamma_0(X))=0
\end{equation}
(see Lemma~\ref{lemma_Poisson_measure_zero}).

On the other hand, the spectral measure $\rho$ of the family $(\tilde
A(\varphi))_{\varphi\in\mathcal D}$ of commuting selfadjoint operators
can be such that
\begin{equation}\label{rho>0}
\rho(\Gamma_0(X))>0,
\end{equation}
where $\Gamma_0(X)$ is a Borel set in the weak topology of the space
$\mathcal{D}'$.  Moreover, property \eqref{rho>0} is used in the
constructions of article \cite{Berez-Tesko-16} (see Theorems 2.10,
6.6), which are essential for Propositions
\ref{prop_A(phi)selfadjoint}, \ref{prop_e}.

We have proved that there is such a family $(\tilde
A(\varphi))_{\varphi\in\mathcal D}$ for which the spectral
measure $\rho$ is equal to the Poisson $\pi$.  Therefore we have some
contradiction with \eqref{rho>0} and \eqref{pi=0}.

But, indeed, this is no contradiction, since our spectral measure is
Poissonian only in the sense of
Theorem~\ref{th_A(phi)_spectral_representation}, i.~e., only by
definition \eqref{Poisson_measure}.

Let us explain the situation in more details.

Definition \eqref{Poisson_measure} is based on the assertion that the
Laplace transform \eqref{Poisson_measure} is defines uniquely a
measure on $\mathcal{D}'$, when this measure is given on some fixed
$\sigma$-algebra on $\mathcal{D}'$ (uniquely up to sets of measure
zero).  But in our case the situation is different.  We define $\pi$
and $\rho$ on $\Gamma(X)$ with a different topology and, therefore, in
principle, on different Borel $\sigma$-algebras.

Namely in \eqref{pi=0} the measure $\pi$ is defined on $\Gamma(X)$
with the ordinary-vague topology \eqref{Gamma_2}, i.~e., on $\Gamma_0(X)$
in the ordinary topology, and on $\Gamma(X)\setminus\Gamma_0(X)$ in the vague
topology (i.~e., the relative topology as on the part
$\Gamma(X)\setminus\Gamma_0(X)$ of $\Gamma(X)$ with the weak topology on
$\mathcal{D}'=\left(C_{\rm fin}^\infty(X)\right)'$).

In \eqref{rho>0} we have another topology on $\Gamma(X)$; this is the
weak topology on $\mathcal{D}'$ with the inclusion
$\Gamma(X)\subset\mathcal{D}'$.  Such type of the topology is
convenient in spectral theory, see \cite{Berez-Tesko-16},
\cite{Berez-Kondr-95}.

We ca also say that a similar situation about \eqref{rho>0} and
\eqref{pi=0} is in the work \cite{Berez-00}, where the spectral
measure of special Jacobi fields may be Poissonian in the sense of
definition \eqref{Poisson_measure}.



\begin{thebibliography}{100000}

\bibitem{Albeverio-98}
S.~Albeverio, Yu.~G. Kondratiev, and M.~R{\"o}ckner,
\emph{Analysis and
  geometry on configuration spaces}, J. Funct. Anal. \textbf{154} (1998),
  no.~2, 444--500.

\bibitem{Berez-93}
Yu.~M. Berezansky, \emph{Spectral approach to white noise
analysis}, Dynamics
  of complex and irregular systems (Bielefeld, 1991), World Sci. Publ., River
  Edge, NJ, 1993, pp.~131--140.

\bibitem{Berez-98}
Yu.~M. Berezansky, \emph{Commutative jacobi fields in fock space},
Integral
  Equations Operator Theory \textbf{30} (1998), no.~2, 163--190.

\bibitem{Berez-00}
Yu.~M. Berezansky, \emph{Poisson measure as the spectral measure
of {J}acobi
  field}, Infin. Dimens. Anal. Quantum Probab. Relat. Top. \textbf{3} (2000),
  no.~1, 121--139.

\bibitem{Berez-03}
Yu.~M. Berezansky, \emph{The generalized moment problem associated
with
  correlation measures}, Funct. Anal. Appl. \textbf{37} (2003), no.~4,
  311--315.

\bibitem{Berez-Kondr-95}
Yu.~M. Berezansky and Y.~G. Kondratiev, \emph{Spectral methods in
  infinite-dimensional analysis. {V}ol.~1}, Kluwer Academic Publishers,
  Dordrecht, 1995.

\bibitem{Berez-Livin-Lytv-95}
Yu.~M. Berezansky, V.~O. Livinsky, and E.~W. Lytvynov, \emph{A
generalization
  of {G}aussian white noise analysis}, Methods Funct. Anal. Topology \textbf{1}
  (1995), no.~1, 28--55.

\bibitem{Berez-Mierz-07}
Yu.~M. Berezansky and D.~A. Mierzejewski, \emph{The investigation
of a
  generalized moment problem associated with correlation measures}, Methods
  Funct. Anal. Topology \textbf{13} (2007), no.~2, 124--151.

\bibitem{Berez-Tesko-16}
Yu.~M. Berezansky and V.~A. Tesko, \emph{The investigation of
bogoliubov
  functionals by operator methods of moment problem}, Methods Funct. Anal.
  Topology \textbf{22} (2016), no.~1, 1--47.

\bibitem{Fink-doctor}
D.~L. Finkelshtein, \emph{Stochastic dynamics of continuous
systems}, Doctor's
  thesis, Institute of Mathematics of NAS of Ukraine, Kyiv, 2014.

\bibitem{Ito-Kubo}
Yoshifusa Ito and Izumi Kubo, \emph{Calculus on {G}aussian and
{P}oisson white
  noises}, Nagoya Math. J. \textbf{111} (1988), 41--84.

\bibitem{Kondr-Kuna}
Yuri~G. Kondratiev and Tobias Kuna, \emph{Harmonic analysis on
configuration
  space. {I}. {G}eneral theory}, Infin. Dimens. Anal. Quantum Probab. Relat.
  Top. \textbf{5} (2002), no.~2, 201--233.

\bibitem{Lytv-95}
E.~W. Lytvynov, \emph{Multiple {W}iener integrals and
non-{G}aussian white
  noises: a {J}acobi field approach}, Methods Funct. Anal. Topology \textbf{1}
  (1995), no.~1, 61--85.

\bibitem{Lytv-15}
Eugene Lytvynov, \emph{The projection spectral theorem and
{J}acobi fields},
  Methods Funct. Anal. Topology \textbf{21} (2015), no.~2, 188--198.

\bibitem{Oliv-phd}
M.~J. Oliveira, \emph{Configuration space analysis and poissonian
white noise
  analysis}, Ph.D. thesis, University of Lisbon, Lisbon, 2002.

\bibitem{Part}
K.~R. Parthasarathy, \emph{Probability measures on metric spaces},
Probability
  and Mathematical Statistics, No. 3, Academic Press, Inc., New York-London,
  1967.

\end{thebibliography}
\end{document}